\documentclass[11pt,a4paper,leqno]{amsart}

\title[Microlocal properties of bisingular operators]{Microlocal properties of bisingular operators}

\author[M. Borsero]{Massimo Borsero}

\address{Dipartimento di Matematica ``Giuseppe Peano'', Universit\`a degli Studi di Torino, Via Carlo Alberto 10, 10123 Torino (TO), Italy.}

\email{massimo.borsero@unito.it}

\author[R. Schulz]{Ren\'e Schulz}

\address{Mathematisches Institut, Georg-August Universit\"at G\"ottingen, Bunsenstra\ss e 3--5, D--37073 G\"ottingen,
Germany}

\email{rschulz@uni-math.gwdg.de}



\usepackage{tikz,xifthen}

\usepackage{latexsym}
\usepackage{amsmath}
\usepackage{amssymb}
\usepackage{amsthm}
\usepackage{amsfonts}
\usepackage{mathrsfs}
\usepackage{calc}
\usepackage{cite}
\usepackage{color}
\usepackage{graphicx}
\usepackage{enumerate}

\usepackage[english]{babel}
\usepackage[latin1]{inputenc}

\usepackage{tikz,xifthen}

\usepackage{hyperref}
\hypersetup{pdfstartview={XYZ null null 1.00}}

\theoremstyle{definition}
\newtheorem{theo}{Theorem}[section]
\newtheorem{lem}{Lemma}[theo]
\newtheorem{coro}{Corollary}[theo]
\newtheorem{ex}[theo]{Example}
\newtheorem{remark}[theo]{Remark}
\newtheorem{deff}[theo]{Definition}
\newtheorem*{deff*}{Definition}
\newtheorem*{theo*}{Theorem}
\newtheorem{pipi}[theo]{Proposition}

\newcommand \erre{{\mathbb{R}}}

\newcommand \ci{{\mathbb{C}}}

\newcommand \zetapiu{{\mathbb{Z}_+^n}}

\newcommand \diprimo{{\mathcal{D}'}}

\newcommand \eeprimo{{\mathcal{E}'}}
\newcommand \esse{\mathscr{S}}

\newcommand \esseprimo{\mathscr{S}'}

\newcommand \RR{{\mathbb{R}}}
\newcommand \NN{{\mathbb{N}}}

\newcommand \WF{{\mathrm{WF}}}
\newcommand \WFcl{{\mathrm{WF}}_\mathrm{cl}}
\newcommand \Sm{{\mathcal{C}^\infty}}
\newcommand \Smc{{\mathcal{C}_0^\infty}}

\newcommand \Char{{\mathrm{Char}}}
\newcommand \Charcl{{\mathrm{Char}}_{\mathrm{cl}}}

\newcommand \normics{\langle x \rangle}
\newcommand \normxi{\langle \xi \rangle}

\newcommand{\semi}[1]{\langle #1 \rangle}

\newcommand {\WFsg}[1]{{\mathrm{WF}_{\mathrm{SG}}^{#1}}}
\newcommand {\WFbi}[1]{{\mathrm{WF}_{\mathrm{bi}}^{#1}}}
\newcommand {\Charbi}[1]{{\mathrm{Char}_{\mathrm{bi}}^{#1}}}

\newcommand{\bisi}[2]{\mathrm{S}^{#1,#2}}
\newcommand{\elbisi}[2]{\mathrm{L}^{#1,#2}}

\newcommand \SG{{\mathrm{SG}}}
\newcommand \SGL{{\mathrm{LG}}}

\newcommand \dist{{\mathcal{D}^\prime}}

\newcommand{\Feyn}[1]{#1\kern-0.45em/}
\newcommand{\dbar}{d\llap {\raisebox{.9ex}{$\scriptstyle-\!$}}}

\DeclareMathOperator{\supp}{\mathrm{supp}}

\begin{document}

\begin{abstract}
\noindent We study the microlocal properties of bisingular operators, a class of operators on the product of two compact manifolds. We define a wave front set for such operators, and analyse its properties. We compare our wave front set with the $SG$ wave front set, a global wave front set which shares with it formal similarities.
\end{abstract}

\maketitle

\section*{Introduction}
\noindent Bisingular operators were originally introduced by L. Rodino in \cite{Rod75} as a class of operators on a product of two compact manifolds $\Omega_1 \times \Omega_2$, defined as linear and continuous operators $A = Op(a) $ whose symbol satisfies, in local product-type coordinates, the estimate
\begin{equation*}
|D^{\alpha_1}_{\xi_1} D^{\alpha_2}_{\xi_2}D^{\beta_1}_{x_1}D^{\beta_2}_{x_2} a(x_1,x_2,\xi_1,\xi_2)|\leq C_{\alpha_1,\alpha_2,\beta_1,\beta_2} \semi{\xi_1}^{m_1-|\alpha_1|} \semi{\xi_2}^{m_2-|\alpha_2|}.
\end{equation*}
A simple and fundamental example of a bisingular operator is the tensor product $A_1 \otimes A_2$ of two pseudodifferential operators, with symbols in the H\"ormander class, $A_i \in  \mathrm{L}^{m_i}(\Omega_i)$, $i=1,2$, while more complex examples include the vector-tensor product $A_1 \boxtimes A_2$ studied in \cite{Rod75}, and the double Cauchy integral operator studied in \cite{NR06}.
To each symbols of a bisingular operator we can associate two maps
\begin{align*}
\sigma^{1} :&\, \Omega_1 \times \erre^{n_1} \rightarrow \mathrm{L}^{m_2}(\Omega_2) \\
\sigma^{2} :&\, \Omega_2 \times \erre^{n_2} \rightarrow  \mathrm{L}^{m_1}(\Omega_1),
\end{align*}  
and with these maps the bisingular calculus takes the form of a calculus with vector valued symbols. General vector valued calculi have been deeply studied, for example, by B. V. Fedosov, B.-W. Schulze and N. N. Tarkhanov in \cite{FST98}. \newline
\noindent In this paper we deal with bisingular operators whose symbols follow H\"{o}rmander-type estimates (see e.g. \cite{Hor83b}), however a global version of bisingular calculus was defined by U. Battisti, T. Gramchev, S. Pilipovi\'{c} and L. Rodino in \cite{BGPR13}. In particular, we only study operators on compact manifolds, given explicitly in local coordinates. We also note that `product-type' operators calculi, similar to bisingular calculus, were introduced by V. S. Pilidi \cite{Pil73}, R. V. Dudu\v{c}ava \cite{Dud79a}, \cite{Dud79b}, and, more recently, by R. Melrose and F. Rochon \cite{MR06}. Moreover, multisingular calculi were considered by V. S. Pilidi \cite{Pil71} and L. Rodino \cite{Rod80}. \newline
\noindent Applications of bisingular calculus include Index Theory, see e.g. \cite{NR06}, Analitic Number Theory, see e.g. \cite{Bat12}, and Geometric Analysis, see e.g. \cite{GH13}. \newline
The aim of this paper is to study the microlocal properties of bisingular operators. In order to do this, we define a suitable wave front set for such operators, called the \textit{Bi-wave front set}, which is the union of three components
	$$\WFbi{}(u)=\WFbi{1}(u)\cup \WFbi{2}(u)\cup \WFbi{12}(u),$$
$u\in\dist(\Omega_1\times\Omega_2)$. This definition is formulated using the calculus only, and, roughly speaking, has this interpretation: $\WFbi{1}$ takes care of the singularities which, in the wave front space, lie in the axis $\xi_2=0$, $\WFbi{2}$ takes care of the singularities which lie in the axis $\xi_1=0$, and $\WFbi{12}$ of the remaining singularities, which include all the classical ones. Our wave front set is related to the classical H\"{o}rmander wave front set $\WFcl$ (see \cite{Hor83a}) via the following inclusion
\begin{equation*}
\WFcl{u}\cap\left(\Omega_1\times\Omega_2\times(\RR^{n_1}\setminus\{0\})\times(\RR^{n_2}\setminus\{0\})\right) \subset \WFbi{12}(u).
\end{equation*}
Moreover, we have a global regularity result
$$\WFbi{}(u)=\emptyset\Leftrightarrow u\in\Sm(\Omega_1\times\Omega_2),$$
which implies that the bi-wave front set encompasses all the singularities. Our main result is the following:
\begin{pipi}
Let $C$ be a bisingular operator, $u\in\dist(\Omega_1\times\Omega_2)$. Then $$\WFbi{}(Cu)\subset\WFbi{}(u).$$
\end{pipi}
\noindent This Proposition shows the our wave front set is microlocal with respect to bisingular operators. Then, we define an appropriate notion of characteristic set for bisingular operators, given again as a union of three components, 
\begin{equation*}
\Charbi{}(C) := \Charbi{1}(C) \cup \Charbi{2}(C) \cup \Charbi{12}(C),
\end{equation*}     
and with this notion we can get a microellipticity result for the $1$- and $2$-components of the bi-wave front set
\begin{pipi}
Let $C$ be a bisingular operator, $u\in\dist(\Omega_1\times\Omega_2)$. Then
\begin{equation*}
\WFbi{i}(u) \subseteq \Charbi{i}(C) \cup \WFbi{i}( C u),
\end{equation*}
$i=1,2$.
\end{pipi}
\noindent We note strong formal similarities between the bisingular calculus and the so called $SG$-calculus, introduced on $\RR^n$ by H. O. Cordes \cite{Cor95} and C. Parenti \cite{Par79}, see also R. Melrose \cite{Mel95}, Y. Egorov and B.-W. Schulze \cite{ES97}, and E. Schrohe \cite{Sch87}. For this reason, our wave front set has formal connections to and similar features as the global $SG$-wave front set, or $\mathscr{S}$-wave front set, introduced by S. Coriasco and L. Maniccia \cite{CM03}, see also R. Melrose \cite{Mel94} for a geometric scattering version. \newline
\noindent The paper is organized as follows. In Section \ref{sec:1} we fix some notation and briefly review the bisingular calculus. In Section \ref{sec:2}, following the ideas in \cite{Hor83a} and \cite{GS94}, we study the mapping properties of bisingular operators and their microlocal properties with respect to the classical wave front set $\WFcl$. In Section \ref{sec:3} we define the bi-wave front set and state the main results concercing microlocality and microellipticity of bisingular operators. In Sections \ref{sec:4} we compare the bisingular calculus and the $SG$ calculus, focusing on the relations and differences between the bi-wave front set and the $SG$ wave front set.
%
\subsection*{Acknowledgements}
%
We are grateful to Profs. D. Bahns, U. Battisti, S. Coriasco, L. Rodino, I. Witt for valuable advice and constructive criticism.
\noindent This work was supported by the German Research Foundation (Deutsche Forschungsgemeinschaft, DFG) through the Institutional Strategy of the University of G\"ottingen, in particular through the research training group GRK 1493 and the Courant Research Center ``Higher Order Structures in Mathematics''. \\
The second author is also grateful for the support received by the Studienstiftung des Deutschen Volkes and the German Academic Exchange Service (DAAD), as part of this collaboration was funded within the framework of a ``DAAD Doktorandenstipendium''.
\section{Preliminaries} \label{sec:1}
\subsection{Introduction to bisingular calculus}
In this section we will recall the main definitions and properties of bisingular symbols and bisingular operators with homogeneous principal symbols. For the H\"{o}rmander pseudodifferential operators, whose tensor products provide the model example of bisingular operators, we use the notations from \cite{Hor83b}. $\Omega_i,\, i=1,2$, denotes an open domain of $\RR^{n_i}$.
\begin{deff}
$\bisi{m_1}{m_2}(\Omega_1,\Omega_2)$ is the set of $\mathcal{C}^\infty(\Omega_1 \times \Omega_2 \times \RR^{n_1} \times \RR^{n_2})$ functions such that, for all multiindex $\alpha_i, \beta_i$ and for all compact subsets $K_i \subset \subset \Omega_i,\,i=1,2$, there exists a constant $C_{\alpha_1,\alpha_2,\beta_1,\beta_2,K_1,K_2} > 0$ such that
\begin{equation*}
|D^{\alpha_1}_{\xi_1} D^{\alpha_2}_{\xi_2}D^{\beta_1}_{x_1}D^{\beta_2}_{x_2} a(x_1,x_2,\xi_1,\xi_2)|\leq C_{\alpha_1,\alpha_2,\beta_1,\beta_2,K_1,K_2} \semi{\xi_1}^{m_1-|\alpha_1|} \semi{\xi_2}^{m_2-|\alpha_2|},
\end{equation*}
for all $x_i \in K_i,\ \xi_i \in \RR^{n_i}$. As usual, $\normxi := (1+|\xi|^2)^{\frac{1}{2}}$. An element of $\mathrm{S}^{m_1,m_2}(\Omega_1,\Omega_2)$ is called a symbol.
\end{deff}
\begin{deff}
A linear operator $A : \mathcal{C}^\infty_0(\RR^{n_1+n_2}) \rightarrow \mathcal{C}^\infty(\RR^{n_1+n_2})$ is called a bisingular operator if it can be written in the form
\begin{equation*}
\begin{split}
A(u)&(x_1,x_2) = Op(a)[u](x_1,x_2)
\\ &=\frac{1}{(2\pi)^{n_1+n_2}}\int_{\RR^{n_1}} \int_{\RR^{n_2}} e^{i(x_1 \cdot \xi_1 + x_2 \cdot \xi_2)}\,a(x_1,x_2,\xi_1,\xi_2)\hat{u}(\xi_1,\xi_2)\,d\xi_1 d\xi_2,
\end{split}
\end{equation*}
where $a \in \bisi{m_1}{m_2}(\Omega_1,\Omega_2)$ and $\hat{u}$ denotes the Fourier transform of $u$. \newline
\noindent $\elbisi{m_1}{m_2}(\Omega_1,\Omega_2)$ denotes the set of all bisingular operators with symbol in $\bisi{m_1}{m_2}(\Omega_1,\Omega_2)$. Moreover, we set 
\begin{equation*}
\begin{split}
\bisi{\infty}{\infty}(\Omega_1,\Omega_2) &:= \bigcup_{m_1,m_2} \bisi{m_1}{m_2}(\Omega_1,\Omega_2) \\
\bisi{-\infty}{-\infty}(\Omega_1,\Omega_2) &:= \bigcap_{m_1,m_2} \bisi{m_1}{m_2}(\Omega_1,\Omega_2)
\end{split}
\end{equation*}
and by  $\elbisi{\infty}{\infty}(\Omega_1,\Omega_2),\, \elbisi{-\infty}{-\infty}(\Omega_1,\Omega_2)$ the corresponding class of operators. The operators in $\elbisi{-\infty}{-\infty}(\Omega_1,\Omega_2)$ are called \textit{smoothing operators}.
\end{deff}
\noindent We associate to every $a \in \bisi{m_1}{m_2}(\Omega_1,\Omega_2)$ the two maps
\begin{align*}
A^{1} :&\, \Omega_1 \times \erre^{n_1} \rightarrow \mathrm{L}^{m_2}(\Omega_2) \\
& (x_1, \xi_1) \mapsto a(x_1,x_2,\xi_1,D_2),\\
A^{2} :&\, \Omega_2 \times \erre^{n_2} \rightarrow  \mathrm{L}^{m_1}(\Omega_1) \\
& (x_2, \xi_2) \mapsto a(x_1,x_2,D_1,\xi_2),\\
\end{align*}
and for $a \in \bisi{m_1}{m_2}(\Omega_1,\Omega_2),\, b \in \bisi{p_1}{p_2}(\Omega_1,\Omega_2)$ we set (for fixed $x_1,\ x_2$ respectively)
\begin{align*}
a \circ_1 b (x_1,x_2,\xi_1, D_2) &:= (A^{1} \circ B^{1}) (x_1,x_2,\xi_1, D_2) \in \mathrm{L}^{m_2+p_2}(\Omega_2) \\
a \circ_2 b (x_1,x_2,D_1, \xi_2) &:= (A^{2} \circ B^{2}) (x_1,x_2,D_1, \xi_2) \in \mathrm{L}^{m_1+p_1}(\Omega_1).
\end{align*}
\begin{deff}
Let $a \in \bisi{m_1}{m_2}(\Omega_1,\Omega_2)$. Then $a$ has a homogeneous principal symbol if
\begin{enumerate}[i)]
\item there exists $a_{m_1 ; \cdot } \in \bisi{m_1}{m_2}(\Omega_1,\Omega_2)$ such that
\begin{align*}
a_{m_1 ; \cdot }(x_1,x_2,t \xi_1,\xi_2) &= t^{m_1} a_{m_1 ; \cdot }(x_1,x_2,\xi_1,\xi_2), \,\,\forall x_1,x_2,\xi_2,\,\, \forall |\xi_1| > 1, t>0, \\
a-\psi_1(\xi_1)a_{m_1 ; \cdot } &\in \bisi{m_1-1}{m_2},
\end{align*}
where $\psi_1$ is an $0$-excision function. 
Moreover, $a_{m_1 ; \cdot } (x_1,x_2,\xi_1,D_2) \in \mathrm{L}^{m_2}_{\mathrm{cl}}(\Omega_2)$, so, being a classical symbol on $\Omega_2$, it admits an asymptotic expansion with respect to the $\xi_2$ variable.
\item there exists $a_{\cdot ; m_2 } \in \bisi{m_1}{m_2}(\Omega_1,\Omega_2)$ such that
\begin{align*}
a_{\cdot ; m_2 }(x_1,x_2,\xi_1,t\xi_2) &= t^{m_2} a_{\cdot ; m_2 }(x_1,x_2,\xi_1,\xi_2), \,\,\forall x_1,x_2,\xi_1,\,\, \forall |\xi_2| > 1, t>0, \\
a-\psi_2(\xi_2)a_{\cdot ; m_2 } &\in \bisi{m_1}{m_2-1},\,\, \psi_2 \text{ as $\psi_1$ above.} 
\end{align*}
Moreover, $a_{\cdot ; m_2 } (x_1,x_2,D_1,\xi_2) \in \mathrm{L}^{m_1}_{\mathrm{cl}}(\Omega_1)$, so, being a classical symbol on $\Omega_1$, it admits an asymptotic expansion with respect to the $\xi_1$ variable.
\item
the symbols $a_{m_1 ; \cdot }$ and $a_{\cdot ; m_2 }$ have the same leading term, so there exists $a_{m_1 ; m_2 }$ such that
\begin{align*}
a_{m_1 ; \cdot } - \psi_2(\xi_2) a_{m_1 ; m_2} &\in \mathrm{S}^{m_1,m_2-1}(\Omega_1,\Omega_2), \\
a_{\cdot ; m_2 } - \psi_1(\xi_1) a_{m_1 ; m_2} &\in \mathrm{S}^{m_1-1,m_2}(\Omega_1,\Omega_2), 
\end{align*}
and
\begin{equation*}
a-\psi_1 a_{m_1 ; \cdot}-\psi_2 a_{\cdot ; m_2} +\psi_1\psi_2 a_{m_1 ; m_2}\in \mathrm{S}^{m_1-1,m_2-1}(\Omega_1,\Omega_2).
\end{equation*}
\end{enumerate}
The symbols which admit a full bi-homogeneous expansion in $\xi_1$ and $\xi_2$ are called \textit{classical symbols}, their class is denoted by $\mathrm{S}^{m_1,m_2}_{\mathrm{cl}}(\Omega_1,\Omega_2)$, and the corresponding operator class by $\mathrm{L}^{m_1,m_2}_{\mathrm{cl}}(\Omega_1,\Omega_2)$.
\end{deff}
\noindent The previous Definition implies that, given $A 	\in \mathrm{L}^{m_1,m_2}_{\mathrm{cl}}(\Omega_1,\Omega_2)$, we can define maps $\sigma^{1},\,\sigma^{2},\,\sigma^{12}$ in this way
\begin{align*}
\sigma^{1}(A) :&\, T^*\Omega_1 \setminus 0 \rightarrow  \mathrm{L}^{m_2}_{\mathrm{cl}}(\Omega_2) \\
& (x_1, \xi_1) \mapsto a_{m_1 ; \cdot} (x_1,x_2,\xi_1,D_2),\\
\sigma^{2}(A) :&\, T^*\Omega_2 \setminus 0 \rightarrow  \mathrm{L}^{m_1}_{\mathrm{cl}}(\Omega_1) \\
& (x_2, \xi_2) \mapsto a_{\cdot ; m_2} (x_1,x_2,D_1,\xi_2),\\
\sigma^{12}(A) :&\, (T^*\Omega_1 \setminus 0) \times (T^*\Omega_2 \setminus 0)  \rightarrow  \ci \\
& (x_1, x_2, \xi_1, \xi_2) \mapsto a_{m_1 ; m_2} (x_1,x_2,\xi_1,\xi_2),
\end{align*}
such that, denoting by $\sigma(P)(x,\xi)$ the principal symbol of a pseudodifferential operator $P$, we have 
\begin{equation*}
\begin{split}
\sigma(\sigma^{1}(A)(x_1,\xi_1) )(x_2,\xi_2) &= \sigma(\sigma^{2}(A)(x_2,\xi_2) )(x_1,\xi_1) \\
&= \sigma^{12}(A)(x_1, x_2, \xi_1, \xi_2) = a_{m_1 ; m_2} (x_1,x_2,\xi_1,\xi_2).  
\end{split}
\end{equation*}
We call the couple $(\sigma^{1}(A), \sigma^{2}(A))$ the \textit{principal symbol} of $A$. \newline
\noindent In the following, we only consider bisingular operators on the product of two compact manifolds $\Omega_1, \Omega_2$. They are defined as above in local coordinates. For those, there exists a notion of ellipticity, called bi-ellipticity. For more details, see \cite{Rod75}.
\begin{deff}\label{def:biell}
Let $A 	\in \mathrm{L}^{m_1,m_2}_{\mathrm{cl}}(\Omega_1,\Omega_2)$. We say that $A$ is bi-elliptic if
\begin{enumerate}
\item[i)] $\sigma^{12}(A)(v_1,v_2) \neq 0$ for all $(v_1,v_2) \in (T^*\Omega_1 \setminus 0) \times (T^*\Omega_2 \setminus 0)$;
\item[ii)]$\sigma^{1}(A)(v_1)$ is exactly invertible as an operator in $\mathrm{L}^{m_2}_{\mathrm{cl}}(\Omega_2)$ for all $v_1 \in T^*\Omega_1 \setminus 0$, with inverse in $\mathrm{L}^{-m_2}_{\mathrm{cl}}(\Omega_2)$;  
\item[iii)]$\sigma^{2}(A)(v_2)$ is exactly invertible as an operator in $\mathrm{L}^{m_1}_{\mathrm{cl}}(\Omega_1)$ for all $v_2 \in T^*\Omega_2 \setminus 0$, with inverse in $\mathrm{L}^{-m_1}_{\mathrm{cl}}(\Omega_1)$. 
\end{enumerate}
\end{deff}
\noindent To elaborate on this definition, we give some examples.
\begin{ex}
Consider the differential operator
\begin{equation*}
A = \sum_{\substack{|\beta_1| \leq m_1 \\ |\beta_2| \leq m_2}} c_{\beta_1 , \beta_2}(x_1,x_2) D^{\beta_1}_1 D^{\beta_2}_2,
\end{equation*}
with $\Sm$ coefficients. In this case
\begin{align}
\sigma^{1}(A)(x_1,\xi_1) &= \sum_{\substack{|\beta_1| = m_1 \\ |\beta_2| \leq m_2}} c_{\beta_1 , \beta_2}(x_1,x_2) \xi^{\beta_1}_1 D^{\beta_2}_2  \label{laprima}\\
\sigma^{2}(A)(x_2,\xi_2) &= \sum_{\substack{|\beta_1| \leq m_1 \\ |\beta_2| = m_2}} c_{\beta_1 , \beta_2}(x_1,x_2) D^{\beta_1}_1 \xi^{\beta_2}_2. \label{laseconda}
\end{align}
\noindent A full bi-homogeneous expansion is given by
\begin{equation*}
\tilde \sigma^{i , j}(A)(x_1,x_2,\xi_1,\xi_2) = \sum_{\substack{|\beta_1| = i \\ |\beta_2| = j}} c_{\beta_1 , \beta_2}(x_1,x_2) \xi^{\beta_1}_1 \xi^{\beta_2}_2.
\end{equation*} 
\noindent The bi-ellipticity of $A$ is given by the condition  $\sigma^{12}(A)=\tilde \sigma^{m_1,m_2}(A)(v_1,v_2) \neq 0$ for all $(v_1,v_2) \in (T^*\Omega_1 \setminus 0) \times (T^*\Omega_2 \setminus 0)$ and the invertibility of the two maps \eqref{laprima} and \eqref{laseconda}. \newline
\noindent We may give a global meaning to $A$ on a product of compact manifolds $\Omega_1 \times \Omega_2$ by taking, for example, $\Omega_j = \mathbb{T}_j$, the $n_j$-dimensional torus, $j=1,2$, and $x_j$ angular coordinates on $\mathbb{T}_j$.
\end{ex}
\noindent With this in mind, it is possible to study some model cases of operators of the form $A\otimes B$. In the following Table \ref{tab:exbisi} we mean by $\Psi$DO the classical pseudodifferential operators on $\Omega_1 \times \Omega_2$, and by $\Psi$DO-order and $\Psi$DO-ell. their order and ellipticity, respectively.
\begin{table}[!ht]
\renewcommand{\arraystretch}{1.2}
\begin{center}
\begin{tabular}{|c|c|c|c|c|}
\hline 
Operator & $\Psi$DO-order & $\Psi$DO-ell. & Bi-order & Bi-ell.  \\
\hline
$I\otimes I$ & $0$ & $\surd$ & $(0,0)$ & $\surd$ \\ 
\hline 
$-\Delta_1 \otimes I+I\otimes (-\Delta_2)$ & $2$ & $\surd$ & $(2,2)$ & $\times$ \\ 
\hline 
$-\Delta_1 \otimes (-\Delta_2) $ & $4$ & $\times$ & $(2,2)$ & $\times$ \\ 
\hline 
$-\Delta_1 \otimes (-\Delta_2+I) $ & $4$ & $\times$ & $(2,2)$ & $\times$ \\  
\hline 
$(-\Delta_1+I) \otimes (-\Delta_2+I) $ & $4$ & $\times$ & $(2,2)$ & $\surd$ \\  
\hline 
$(-\Delta_1+I)^{-1} \otimes (-\Delta_2+I)^{-1} $ & not a $\Psi$DO &  & $(-2,-2)$ & $\surd$ \\ 
\hline 
\end{tabular} 
\end{center}
\label{tab:exbisi}
\caption{Some model cases of bisingular operators of tensor product type}
\end{table}
\begin{theo}
Let $A 	\in \mathrm{L}^{m_1,m_2}_{\mathrm{cl}}(\Omega_1,\Omega_2)$ be bi-elliptic. Then there exists $B \in \mathrm{L}^{-m_1,-m_2}_{\mathrm{cl}}(\Omega_1,\Omega_2)$ such that
\begin{align*}
AB &= I + K_1\\
BA &= I + K_2,
\end{align*}
where $I$ is the identity map and $K_1,\,K_2$ are smoothing operators. Moreover the principal symbol of $B$ is $(\sigma^{1}(A)^{-1}, \sigma^{2}(A)^{-1})$.
\end{theo}
\noindent From now on we will assume, for simplicity, that symbols of bisingular operators have compact support in the $x_1, x_2$ variables.
\begin{theo}\label{thm:compos}
Let $a \in \bisi{m_1}{m_2}(\Omega_1,\Omega_2), b \in \bisi{p_1}{p_2}(\Omega_1,\Omega_2)$. Then $AB \in \mathrm{L}^{m_1+p_1,m_2+p_2}(\Omega_1,\Omega_2)$, and its symbol $c(x_1,x_2,\xi_1,\xi_2)$ has the asymptotic expansion
\begin{equation*}
c \sim \sum_{j=0}^{\infty} c_{m_1+p_1-j,m_2+p_2-j}
\end{equation*}
where
\begin{equation*}
\begin{split}
c_{m_1+p_1-j,m_2+p_2-j} =& \ c^1_{m_1+p_1-j-1,m_2+p_2-j} + c^2_{m_1+p_1-j,m_2+p_2-j-1} \\
&+ c^{12}_{m_1+p_1-j,m_2+p_2-j}
\end{split}
\end{equation*}
and
\begin{align*}
c^1&_{m_1+p_1-j-1,m_2+p_2-j}\\ 
&= \sum_{|\alpha_2| = j} \frac{1}{\alpha_2!} \left\{ \partial^{\alpha_2}_{\xi_2} a \circ_1 D^{\alpha_2}_{x_2} b - \sum_{|\alpha_1| \leq j} \frac{1}{\alpha_1 !} \partial^{\alpha_1}_{\xi_1} \partial^{\alpha_2}_{\xi_2} a\, D^{\alpha_1}_{x_1} D^{\alpha_2}_{x_2} b \right\} \\
c^2&_{m_1+p_1-j,m_2+p_2-j-1} \\
&= \sum_{|\alpha_1| = j} \frac{1}{\alpha_1!} \left\{ \partial^{\alpha_1}_{\xi_1} a \circ_2 D^{\alpha_1}_{x_1} b - \sum_{|\alpha_2| \leq j} \frac{1}{\alpha_2 !} \partial^{\alpha_1}_{\xi_1} \partial^{\alpha_2}_{\xi_2} a\, D^{\alpha_1}_{x_1} D^{\alpha_2}_{x_2} b \right\} \\
 c^{12}&_{m_1+p_1-j,m_2+p_2-j} \\
 &= \sum_{|\alpha_1| = |\alpha_2| = j}  \frac{1}{\alpha_1 ! \alpha_2 !} \partial^{\alpha_1}_{\xi_1} \partial^{\alpha_2}_{\xi_2} a\, D^{\alpha_1}_{x_1} D^{\alpha_2}_{x_2} b
\end{align*}
\end{theo}

\begin{coro} \label{coro:comm}
Let $a \in \bisi{m_1}{m_2}(\Omega_1,\Omega_2), b \in \bisi{p_1}{p_2}(\Omega_1,\Omega_2)$. Then the commutator $[A,B] := AB - BA$ belongs to $\mathrm{L}^{m_1+p_1-1, m_2+p_2} + L^{m_1+p_1, m_2+p_2-1}$.
\begin{proof}
By Theorem \ref{thm:compos} we have as leading order terms $(j=0)$:
\begin{align*}
&c^1_{m_1+p_1-1,m_2+p_2} = a \circ_1 b -   b \circ_1 a \\
&c^2_{m_1+p_1,m_2+p_2-1} = a \circ_2 b -   b \circ_2 a\\ 
&c^{12}_{m_1+p_1,m_2+p_2} = 0.
\end{align*}
Then, expanding $c^1$ and $c^2$ according to the definition of $\circ_j,\ j=1,2$, we get
\begin{align*}
c^j = i\{a,b\}_j + \text{terms of order $(m_j+p_j-2)$,}
\end{align*}
where with $\{a,b\}_j$ we denote the Poisson bracket of $a$ and $b$ in the $j$-argument. Therefore, the leading terms (up to order $(m_1+p_1-2,m_2+p_2-2)$) of the expansion of $c$ can be written as
\begin{align*}
c = 0 + i(\{a,b\}_1 + \{a,b\}_2) \in \mathrm{S}^{m_1+p_1-1, m_2+p_2} + \mathrm{S}^{m_1+p_1, m_2+p_2-1}.
\end{align*}
\end{proof}
\end{coro}
\begin{remark}
This behaviour under commutators is indeed something peculiar about bisingular calculus. It has the consequence that we can not use a lot of the common ``commutator tricks" in the proofs to obtain microlocal properties. 
\end{remark}
\begin{ex}
\label{ex:comm}
For a better understanding of this phenomenon, consider the model case of a tensor product where $A = A_1 \otimes A_2 \in \elbisi{m_1}{m_2}$, $B = B_1 \otimes B_2 \in \elbisi{p_1}{p_2}$. Then
\begin{equation*}
\begin{split}
[A, B] = [A_1 \otimes A_2 , B_1 \otimes B_2] &= A_1 B_1 \otimes A_2 B_2 - B_1 A_1 \otimes B_2 A_2 \\
&= \underbrace{[A_1, B_1] \otimes A_2 B_2}_{\elbisi{m_1 +p_1 -1}{m_2+ p_2}} + \underbrace{B_1 A_1 \otimes [A_2 , B_2]}_{\elbisi{m_1 +p_1}{m_2 +p_2 -1}}.   
\end{split}
\end{equation*}
\end{ex}

\noindent To close this section, we note that all pseudodifferential operators of order zero or lower, in particular those corresponding to cut-offs and excision functions, are bisingular operators.

\begin{lem} \label{lem:inclu}
$\mathrm{S}^0(\Omega_1 \times \Omega_2) \subset \bisi{0}{0}(\Omega_1 ,\Omega_2)$.
\begin{proof}
Let $a \in \mathrm{S}^0(\Omega_1 \times \Omega_2)$. Then for all pair of multiindex $\alpha = (\alpha_1, \alpha_2),\,\beta=(\beta_1,\beta_2)$ we have
\begin{equation*}
\begin{split}
|D^{\alpha_1}_{\xi_1} D^{\alpha_2}_{\xi_2}D^{\beta_1}_{x_1}D^{\beta_2}_{x_2} a(x_1,x_2,\xi_1,\xi_2)| &= |D^{\alpha}_{\xi} D^{\beta}_{x} a(x,\xi)| \prec \semi{\xi}^{-|\alpha|}\\
& \prec \semi{\xi}^{-|\alpha_1|} \semi{\xi}^{-|\alpha_2|} \prec \semi{\xi_1}^{-|\alpha_1|} \semi{\xi_2}^{-|\alpha_2|},  
\end{split}
\end{equation*}
that is $a \in \bisi{0}{0}(\Omega_1 ,\Omega_2)$.
\end{proof}
\end{lem}
\section{Classical microlocal analysis of bisingular operators} \label{sec:2}
\noindent Let us first recall the notion of \emph{classical wave front set}, as introduced by H\"ormander \cite{Hor83a}. 
\begin{deff}
Let $\Omega\subset\RR^n$ open, $u\in\dist(\Omega)$. The distribution $u$ is microlocally $\mathcal{C}^\infty$ near $(x_0,\xi_0)\in\mathrm{T}^*\Omega\setminus0$ if one of the following equivalent conditions is satisfied:
\begin{enumerate}
\item There exists a pseudodifferential operator $A\in \mathrm{L}^0(\Omega)$, non-characteristic at $(x_0,\xi_0)$, such that $Au\in\mathcal{C}^\infty(\Omega)$,
\item There exists a cut-off $\phi\in\mathcal{C}^\infty_0 (\Omega)$ with $\phi\equiv 1$ in an open set containing $x_0$ such that there exists a conic open set $\Gamma\subset\RR^n\setminus 0$ containing $\xi_0$ and constants $C_N,\ R>0$ such that $\forall\ N\in\NN$
\[ 
|\widehat{\phi u}(\xi)|\leq C_N \langle \xi\rangle^{-N} \quad \forall\xi\in\Gamma,\ |\xi|>R
\]
\end{enumerate}
The classical wave front set of $u \in \diprimo(\Omega)$, that we denote by $\WFcl(u)$, is the complement of the set of points where $u$ is microlocally $\Sm$.
\end{deff}
\noindent It is now interesting to compare, for given operators $A$, the sets $\WFcl(Au)$ and $\WFcl(u)$. An operator is \emph{microlocal}, if $\WFcl(Au)\subset\WFcl(u)$. 
\subsection{Mapping properties of bisingular operators}
In this Section we shall estimate the classical wave front set $\WFcl(Au)$ for a linear operator $A$ and a distribution $u$ in terms of the Schwartz kernel $K_A$ of $A$, defined as follows:
\begin{deff}
To each operator $A: \mathcal{D}(\Omega_1)\rightarrow \mathcal{D}^\prime(\Omega_2)$ we can uniquely associate a distribution $K_A$, called the \emph{Schwartz kernel} of $A$, such that for all $f\in\mathcal{D}(\Omega_1)$, $g\in\mathcal{D}(\Omega_2)$ we have
\[
\langle Af,\ g\rangle=\langle K_A , f\otimes g\rangle.
\]
\end{deff}
\noindent The Schwartz kernel of a bisingular operator with symbol $a$ is then defined by the oscillatory integral
\begin{equation}
K_A(x_1,x_2,y_1,y_2)=\int_{\RR^{n_1+n_2}} e^{i(x_1-y_1,x_2-y_2)\cdot(\xi_1,\xi_2)} a(x_1,x_2,\xi_1,\xi_2)\dbar\xi_1\dbar\xi_2.
\end{equation}
\noindent
The Schwartz kernel Theorem states the following smoothing property:
\begin{pipi}
A linear map $\mathcal{D}^\prime(\Omega_1\times\Omega_2)\rightarrow\mathcal{D}^\prime(\Omega_1\times\Omega_2)$ is actually a mapping to $\mathcal{D}(\Omega_1\times\Omega_2)$, i.e. \textit{smoothing}, if and only if its distributional kernel is in $\mathcal{D}((\Omega_1\times\Omega_2)^{\times 2})$.
\end{pipi} 
\noindent Therefore pseudodifferential operators with symbols in $S^{-\infty}$ and  bisingular operators with symbol in $S^ {-\infty,-\infty}$ can be seen to be smoothing.\\
As the prototype of a bisingular operator is the tensor product of two pseudodifferential operators, it makes sense also to define what is meant by an operator that is smoothing in one set of variables only.
\begin{deff}
We define the space $\mathcal{C}^\infty(\Omega_1,\diprimo(\Omega_2))$ as all such $u\in\diprimo(\Omega_1\times\Omega_2)$ such that for each $f\in\mathcal{D}(\Omega_2)$, the distribution $\mathcal{D}(\Omega_1)\ni u(g\otimes\cdot):g\mapsto u(g\otimes f)$ is actually a smooth function.\\
Correspondingly, we can define $\mathcal{C}^\infty(\Omega_2,\diprimo(\Omega_1))$.
\end{deff}
\noindent We can now list the mapping properties of bisingular operators on these spaces, following the ideas in \cite{Tre67}.
\begin{lem}
\label{lem:sminonevar}
Bisingular operators map the spaces $\mathcal{C}^\infty(\Omega_1,\diprimo(\Omega_2))$ and $\mathcal{C}^\infty(\Omega_2,\diprimo(\Omega_1))$ into themselves.\\
Let $a\in S^{-\infty,m}$, then the bisingular operator $Op(a)$ maps $\diprimo(\Omega_1\times\Omega_2)$ to $\mathcal{C}^\infty(\Omega_1,\diprimo(\Omega_2))$ and $\mathcal{C}^\infty(\Omega_2,\diprimo(\Omega_1))$ to $\mathcal{C}^\infty(\Omega_1\times\Omega_2).$\\
Accordingly, let $a\in S^{m,-\infty}$, then the bisingular operator $Op(a)$ maps $\diprimo(\Omega_1\times\Omega_2)$ to $\mathcal{C}^\infty(\Omega_2,\diprimo(\Omega_1))$ and $\mathcal{C}^\infty(\Omega_1,\diprimo(\Omega_2))$ to $\mathcal{C}^\infty(\Omega_1\times\Omega_2).$
\end{lem}
\noindent The following Lemma (see e.g. \cite{GS94}) indicates how the singularities of a distribution transform under the action of a linear operator in terms of the singularities of its kernel.
\begin{lem}\label{lem:kerwf}
Let $K \subset \diprimo (\Omega_1 \times \Omega_2)$, and denote by the same letter the corresponding operator $K: \mathcal{C}_0^\infty(\Omega_2) \rightarrow \diprimo (\Omega_1)$. Set
\begin{equation*}
\begin{split}
&\WF'(K)  := \{(x_1,x_2,\xi_1,-\xi_2) \in T^*(\Omega_1 \times \Omega_2) \setminus 0; (x_1,x_2,\xi_1,\xi_2) \in \WFcl(K) \} \\
&\WF'_{\Omega_1}(K) := \{(x_1,\xi_1) \in T^*\Omega_1 \setminus 0 ;\, \exists y \in \Omega_2\,\text{with}\, (x_1, y, \xi_1, 0) \in \WF'(K) \} \\
&\WF'_{\Omega_2}(K) := \{(x_2,\xi_2) \in T^*\Omega_2 \setminus 0 ;\, \exists x \in \Omega_1\,\text{with}\, (x, x_2, 0, \xi_2) \in \WF'(K) \}. 
\end{split}
\end{equation*}
\noindent Then if $u \in \eeprimo(\Omega_2)$ and $\WFcl(u) \cap \WF'_{\Omega_2}(K) = \emptyset$, we have
\begin{equation*}
\WFcl(K u) \subset \WF'(K)(\WFcl(u)) \cup \WF'_{\Omega_1}(K),
\end{equation*}
where we considered $\WF'(K)$ as a relation $T^* \Omega_2 \rightarrow T^* \Omega_1$.
\end{lem}
\subsection{Classical microlocality properties of bisingular operators}
From the previous Lemma it is easy to derive that all pseudodifferential operators are microlocal. In fact, if $K$ is the kernel of a pseudodifferential operator $A$ on $\Omega = \Omega_1 = \Omega_2$, then $\WF'_{\Omega_1}(K) = \emptyset = \WF'_{\Omega_2}(K)$, $\WF'(K) = \{(x,x,\xi,-\xi)\}$, hence from Lemma \ref{lem:kerwf} we get $\WFcl(Au) \subset \WFcl(u)$.
We now study the microlocal properties of bisingular operators.

\begin{ex}\label{ex:conv}
Consider $\Omega_1=\Omega_2=\RR$.
We further pick positive $\phi,\ \psi\in\mathcal{C}^\infty_0(\RR)$.
Now define the pseudodifferential operator $T_\phi$ on
$\mathcal{C}_0^\infty\left(\RR\right)$ by
\begin{equation*}
T_\phi(f):=\phi * f.
\end{equation*}
Then the operator $A:= T_\phi\otimes I$ is a tensor product of two pseudodifferential operators and thus a bisingular operator on $\RR^2$.\\
Now consider the distribution $u=\psi\otimes\delta$. It has the following
wave front set:
\begin{equation*}
\WFcl(u)=\left\{(x_1,0,0,\xi_2)|\,x_1\in\supp(\psi),\ \xi_2\in \RR\setminus 0\right\}.
\end{equation*}
\noindent Then it is easy to see that
\begin{equation*}
\WFcl(A u)=\left\{(x_1,0,0,\xi_2)|\,x_1\in\big(\supp(\psi)+\supp(\phi)\big),\ \xi_2\in \RR\setminus 0\right\}.
\end{equation*}
The example can be similarly given in local coordinates on a product of two compact manifolds. We restricted ourselves to the Euclidean space for the sake of comprehensibility.
\end{ex}
\noindent The previous example shows that, in general, bisingular operators do not have the microlocal property. As a motivation, we start with the model case of a tensor product of two pseudodifferential operators. For that we use the well-known estimate for the wave front set of a tensor product of distributions (cf. e.g. \cite{Hor83a}):

\begin{lem}
Let $u\in\dist(\Omega_1)$, $v\in\dist(\Omega_2)$. Then 
\begin{align*}\
\WFcl(u\otimes v)& \subset \WFcl(u)\times\WFcl(v) \cup (\supp(u)\times\{0\})\times\WFcl(v) \\
&\cup\WFcl(v)\times(\supp(v)\times\{0\}).
\end{align*}
\label{lem:wftens}
\end{lem}

\noindent This can be used to estimate the wave front relation for a bisingular operator given as the tensor product of two pseudodifferential operators. As a matter of fact this extends by direct calculation using standard techniques of integral regularization (see, e.g.  \cite{Hor83a}, \cite{Shu01}, \cite{GS94}) to
\begin{theo}\label{theo:wfkern}
Let $A\in \elbisi{m_1}{m_2}(\Omega_1,\Omega_2)$. Then we can estimate the wave front set of the corresponding kernel as follows:
\begin{equation*}
\begin{split}\label{eq:wfkerntotal}
\WFcl(\mathcal{K}_A)&\subset \{(x_1,x_2,x_1,x_2, \xi_1,\xi_2, -\xi_1,-\xi_2)|x_i\in\Omega_i,\ \xi_i\in \RR^{n_i} \setminus 0\}\\
 &\cup \{(x_1,x_2,x_1,y_2,\xi_1,0,-\xi_1,0)|x_i, y_i\in\Omega_i,\ \xi_1\in \RR^{n_1} \setminus 0\}\\
 &\cup \{(x_1,x_2,y_1,x_2,0,\xi_2, 0,-\xi_2)|x_i, y_i\in\Omega_i,\ \xi_2\in \RR^{n_2} \setminus 0\}.
\end{split}
\end{equation*}
\end{theo}

\begin{coro}
Let $A \in \elbisi{m_1}{m_2}(\Omega_1, \Omega_2)$, $u\in\mathcal{E}^\prime(\Omega_1\times\Omega_2)$. Then
\begin{equation*} 
\begin{split}\label{eq:bswf} 
\WFcl(A u) \subset \WFcl(u) & \cup \{(x_1,x_2,0,\xi_2); \exists y_1 \in \Omega_1 : (y_1,x_2,0,\xi_2) \in \WFcl(u) \} \\
&\cup \{(x_1,x_2,\xi_1,0) ; \exists y_2 \in \Omega_2 : (x_1,y_2,\xi_1,0) \in \WFcl(u)\}.
\end{split}
\end{equation*}
\begin{proof}
Use Theorem \ref{theo:wfkern} and Lemma \ref{lem:kerwf}.
\end{proof}
\end{coro}
\noindent Based on this observation we find that bisingular operators are microlocal with respect to a modified version of the classical wave front set. It is obtained by dropping the information about the precise location of the singularities with the covariable $\xi_1=0$ or $\xi_2=0$ in the corresponding variable.
\begin{pipi} \label{prop:bswf}
Let $A \in \elbisi{m_1}{m_2}(\Omega_1, \Omega_2)$, $u\in\mathcal{E}^\prime(\Omega_1\times\Omega_2)$. Then
\begin{equation*} \label{deff:bswf} 
\begin{split}
\WFcl(Au) &\subset \widetilde{\WFcl}(u) \\
&:= \{(x_1,x_2,\xi_1,\xi_2) : (x_1,x_2,\xi_1,\xi_2) \in \WFcl(u) ; |\xi_1|\,|\xi_2| \neq 0\} \\
\qquad &\cup \{(x_1,x_2,\xi_1,0) ; \exists y_2 \in \Omega_2 : (x_1,y_2,\xi_1,0) \in \WFcl(u)\}\\ 
\qquad &\cup \{(x_1,x_2,0,\xi_2); \exists y_1 \in \Omega_1 : (y_1,x_2,0,\xi_2) \in \WFcl(u) \}.
\end{split}
\end{equation*}
\end{pipi}
\begin{pipi}\label{pipi:equal}
Let $A 	\in \mathrm{L}^{m_1,m_2}_{\mathrm{cl}}(\Omega_1,\Omega_2)$ be bi-elliptic. Then $$\widetilde{\WFcl}(Au) = \widetilde{\WFcl}(u).$$
\begin{proof}
We already know that $\widetilde{\WFcl}(Au)\subset\widetilde{\WFcl}(u)$.\\
 Now, $A 	\in \mathrm{L}^{m_1,m_2}_{\mathrm{cl}}(\Omega_1,\Omega_2)$, therefore there exists $B 	\in \mathrm{L}^{-m_1,-m_2}_{\mathrm{cl}}(\Omega_1,\Omega_2)$ such that $BA - I = K \in  \mathrm{L}^{-\infty,-\infty}(\Omega_1,\Omega_2)$. Thus
\begin{equation*}
\begin{split}
\widetilde{\WFcl}(u) := \widetilde{\WFcl} ((BA - K) u ) &\subset \widetilde{\WFcl}( BA u) \cup \widetilde{\WFcl}( Ku) \\
& \subset \widetilde{\WFcl}(Au) \cup \emptyset = \widetilde{\WFcl}(Au).
\end{split}
\end{equation*} 
\end{proof}
\end{pipi}
\noindent This notion encourages to study microlocal properties of bisingular operators with respect to this modified notion of $\widetilde{\WF}(u)$. However, all of the previous results have been obtained by the study of distribution kernels. It is far more desirable to have a notion of singularities in terms of the actual calculus. This will be provided in the next section.


\section{A wave front set in terms of bisingular operators}\label{sec:3}


\subsection{Microlocal properties of bisingular operators}


While being the description that naturally arises when analysing the kernels of bisingular operators, the notion of wave front set used in the previous section has several drawbacks: it is defined in terms of the H\"ormander wave front set, so in order to calculate it one first has to find that set and then ``forget information". Also, it is not defined in terms of the bisingular calculus, but indeed with respect to the pseudodifferential one.\\
In the following we establish a second notion that does not have these drawbacks. In fact it turns out to be quite similar to the notion introduced in \cite{CM03} for the $SG$-calculus. This is not surprising, as there is a strong similarity in the formulas arising in both calculi. \newline
\noindent From now on, all the pseudodifferential operators will be assumed as properly supported, and all the bisingular operators to be classical.
\begin{deff} \label{def:wft}
Let $u\in\dist(\Omega_1\times\Omega_2)$. We define $\WFbi{}(u)\subset\Omega_1\times\Omega_2\times(\RR^{n_1+n_2}\setminus 0) $ 
as 
	$$\WFbi{}(u)=\WFbi{1}(u)\cup \WFbi{2}(u)\cup \WFbi{12}(u)$$
where
\begin{itemize}
	\item $(x_1,x_2,\xi_1,0)$ is not in $\WFbi{1}(u)$ if there exists an $A\in \mathrm{L}^0_\mathrm{cl}(\Omega_1)$ non-characteristic at $(x_1,\xi_1)$ such that 
	\begin{equation} \label{eq:wft1}
	(A\otimes I)u\in\Sm(\Omega_1,\diprimo(\Omega_2)).
	\end{equation}
	\item $(x_1,x_2,0,\xi_2)$ is not in $\WFbi{2}(u)$ if there exists an $A\in \mathrm{L}^0_\mathrm{cl}(\Omega_2)$ non-characteristic at $(x_2,\xi_2)$ such that 
	\begin{equation} \label{eq:wft2}	
	(I\otimes A)u\in\Sm(\Omega_2,\diprimo(\Omega_1)).
	\end{equation}
	\item $(x_1,x_2,\xi_1,\xi_2)$, $|\xi_1||\xi_2|\neq 0$, is not in $\WFbi{12}(u)$ if there exists a $A_i \in \mathrm{L}^0_\mathrm{cl}(\Omega_i)$, non-characteristic at $(x_i,\xi_i)$, $i=1,2$, such that 
	\begin{align} 
	(A_1 \otimes A_2) u&\in\Sm(\Omega_1\times\Omega_2) \label{eq:wft31}\\
	(A_1 \otimes I) u&\in\Sm(\Omega_1,\diprimo(\Omega_2)) \label{eq:wft32},\\ 
	(I \otimes A_2) u&\in\Sm(\Omega_2,\diprimo(\Omega_1)) \label{eq:wft33}.
  \end{align}
\end{itemize}
\end{deff}
\begin{remark}
\label{rem:123indep}
Note that the conditions \eqref{eq:wft32} and \eqref{eq:wft33} do not follow from \eqref{eq:wft31}: Take $u\in\diprimo(\RR\times\RR)$, $u=\delta(x-1)\delta(y+1)$ and $\psi$ smooth such that $\psi\equiv 1$ for $x>1/2$ and $\psi\equiv 0$ for $x\leq 0$. Then $(\psi(x)\otimes\psi(y))u=0$, as $(1,-1)\notin\supp(\psi\otimes\psi)$. However, for $g\in\mathcal{D}(\RR)$ with $g(-1)\neq 0$ we have $(\psi(x)\otimes I)u(.\otimes g)=\delta(x-1)g(-1)$, which is not smooth.
\end{remark}

\begin{remark}
For a distribution of the form $u = u_1 \otimes u_2$, we have that $\WFbi{12}(u) = \WFcl(u_1) \times \WFcl(u_2)$.
\end{remark}
\noindent In fact we have the following inclusion result:
\begin{lem}
\label{lem:wfcl}
If a point $(x_1,x_2,\xi_1,\xi_2)$, $|\xi_1||\xi_2|\neq 0$, is not in $\WFbi{12}(u)$, then it is not in $\WFcl(u)$.
\end{lem}
\begin{proof}
The proof is a variant of \cite{Hor91}, Proposition 2.8. By definition there exists $A:= A_1 \otimes A_2 \in \mathrm{L}^{0,0}(\Omega_1 \times \Omega_2)$, with $\sigma^{12}(A)(x_1,x_2,\xi_1,\xi_2) \neq 0$, such that $A u\in\Sm(\Omega_1\times\Omega_2)$. Now take a ($\Psi$DO symbol) $\psi(x_1,x_2,\xi_1,\xi_2)\in \mathrm{L}^0_\mathrm{cl}(\Omega_1\times\Omega_2)$ such that
\begin{itemize}
\item on the support of $\psi$ we have %
$\langle \xi_1\rangle\lesssim \langle \xi_2\rangle\lesssim \langle \xi_1\rangle$
\item $\psi$ is non-characteristic at $(x_1,x_2,\xi_1,\xi_2)$. 
\end{itemize}
Then the (bi-singular) operator product $$B:=\psi(x_1,x_2,D_1,D_2)\circ A(x_1, x_2,D_1, D_2)$$
yields a pseudodifferential operator of order 0, plus a smoothing remainder, by virtue of the above inequality on the support of $\psi$ and the symbol expand in Theorem \ref{thm:compos}. It has the following properties:
\begin{itemize}
\item its principal symbol is $\psi\cdot \sigma^{12}(A)$, and thus is non-characteristic (in the sense of $\Psi$DOs) at $(x_1,x_2,\xi_1,\xi_2)$ and of zero order.
\item $Bu= \psi(x_1,x_2,D_1,D_2) A(x_1,x_2,D_1,D_2)u\in \Sm$
\end{itemize}
This proves the claim.
\end{proof}
\begin{remark}
This lemma asserts that in the conic region where both covariables are non-vanishing we can pass from bisingular to pseudodifferential calculus by multiplying by a $\Psi$DO. This has the consequence that the two operator classes have similar microlocal properties (with respect to the classical wave front set) in that region. 
\end{remark}
\noindent The following Lemma gives a similar interpretation to the remaining components, illustrating the loss of localization of singularities already encountered in the previous section.

\begin{lem}
\label{lem:wfclproj}
Let $u\in\eeprimo(\Omega_1\times\Omega_2)$, $(x_1^0,\xi_1^0)\in\Omega_1\times(\RR^{n_1}\setminus 0)$. If for all $y\in\Omega_2$ we have $(x_1^0,y,\xi_1^0,0)\notin\WFcl(u)$ then for all $y\in\Omega_2$ we have $(x_1^0,y,\xi_1^0,0)\notin\WFbi{1}(u)$.\\
Similarly, if for all $x\in\Omega_1$ we have $(x,x_2^0,0,\xi_2^0)\notin\WFcl(u)$ then for all $x\in\Omega_1$ we have $(x,x_2^0,0,\xi_2^0)\notin\WFbi{2}(u)$.

\begin{proof}
We prove the claim for $\WFbi{1}$. Take $(x_1^0,y,\xi_1^0,0)\notin\WFcl(u)$. Then there exist a cut-off $\phi\in\Smc(\Omega_1\times\Omega_2)$ and a conic localizer $\psi_y\in\Sm(\RR^{n_1+n_2})$, non-vanishing in a (conic) neighbourhood $(x_1^0,y)$ and $(\xi_1^0,0)$ respectively, such that 
$$\psi_y(\xi_1,\xi_2)\mathcal{F}_{(x_1,x_2)\mapsto(\xi_1,\xi_2)}\{\phi(x_1,x_2) u\}\in\esse(\RR^{n_1+n_2}).$$
As this holds true for any $y$, and due to compactness of the support of $u$, there exists a cut-off $\phi_1\in\Smc(\Omega_1)$ such that for some conic localizer $\psi\in\Sm(\RR^{n_1+n_2})$
$$\psi(\xi_1,\xi_2)\mathcal{F}_{(x_1,x_2)\mapsto(\xi_1,\xi_2)}\{\phi(x_1) u\}\in\esse(\RR^{n_1+n_2}).$$
This means we can find a conic localizer $\psi_1\in\Sm(\RR^{n_1})$, non-vanishing around $\xi_1^0$, such that 
$\psi(\xi_1)\mathcal{F}_{(x_1,x_2)\mapsto(\xi_1,\xi_2)}\{\phi(x_1) u\}\in\Sm(\RR^{n_1+n_2})$, rapidly decaying with respect to the first variable $\xi_1$ for fixed $\xi_2$, and polynomially bounded everywhere, by the Paley-Wiener-Schwartz Theorem.\\
We define $A\in\mathrm{L}_\mathrm{cl}^0(\Omega_1)$ as the operator 
$$Av(y_1)=\mathcal{F}^{-1}_{\xi_1\mapsto y_1}(\psi(\xi_1)\mathcal{F}_{x_1\mapsto\xi_1}\{\phi(x_1) u\}.$$
By the assumptions on $\phi_1$ and $\psi_1$, $A$ is non-characteristic in the sense of pseudodifferential operators at $(x_1^0,\xi_1^0).$
But for any $f\in\mathcal{D}(\Omega_2)$ we have that $[(A\otimes I) u](f)(x_1)=\mathcal{F}^{-1}_{\xi_1\mapsto x_1}\langle \psi(\xi_1)\mathcal{F}_{(y_1,y_2)\mapsto(\xi_1,\xi_2)}\{\phi(y_1) u\},\widehat{f}\rangle$ is a smooth function, which means $(A\otimes I) u\in\mathcal{C}^\infty(\Omega_1,\diprimo(\Omega_2))$.
\end{proof}
\end{lem}
\begin{pipi}
\label{pipi:regu}
Let $u\in\mathcal{E}^\prime(\Omega_1\times\Omega_2)$. Then $$\WFbi{}(u)=\emptyset\Leftrightarrow u\in\Sm(\Omega_1\times\Omega_2).$$
\end{pipi}
\begin{proof}
Assume $\WFbi{}(u)=\emptyset$. Then, by virtue of Lemma \ref{lem:wfcl}, we have $\WFcl(u)\cap\{(x_1,x_2,\xi_1,\xi_2)|:|\xi_1||\xi_2|\neq 0\}=\emptyset$. Thus $\widehat{u}$ is rapidly decaying on any ray $\RR\cdot(\xi_1,\xi_2)$ where $|\xi_1||\xi_2|\neq 0.$\\
As also $\WFbi{1}(u)=\emptyset$ we can find for each $(x_1,x_2,\xi_1,\xi_2)$ an $A\in \mathrm{L}^0_\mathrm{cl}(\Omega_1)$ non-characteristic at $(x_1,\xi_1)$ such that, by Lemma \ref{lem:sminonevar}, for any  $B\in\mathrm{L}^{m,-\infty}(\Omega_1\times\Omega_2)$ we have $B(I\otimes A)u\in\Sm(\Omega_1\times\Omega_2)$. By compactness and a parametrix construction we can thus conclude that for all $ B\in\mathrm{L}^{m,-\infty}(\Omega_1\times\Omega_2)$ we get $B u\in\Sm(\Omega_1\times\Omega_2).$
Now pick $\phi\in \mathcal{C}_0^{\infty}(\Omega_1\times\Omega_2)$ with $\phi\equiv 1$ on the support of $u$, and define $b(x_1,x_2,\xi_1,\xi_2)=\phi(x_1,x_2)f(\xi_2)$, with $f\in \esse(\RR^{d_2})$. Then 
$$\esse(\RR^{d_1+d_2})\ni\mathcal{F}(Bu)=\mathcal{F}\big(b(x_1,x_2,D_1,D_2)u\big)=(1\otimes f)\widehat{u}.$$
As $f$ was arbitrary and rapidly decaying, this means that $\widehat{u}$ must already be rapidly decaying in the first variable. Repeating the argument for the second variable proves the assertion.
\end{proof}
\noindent Using a parametrix construction, we get by the same arguments:
\begin{pipi}
Let $u\in\mathcal{E}^\prime(\Omega_1\times\Omega_2)$. Then 
\begin{align*}
\WFbi{1}(u)&=\emptyset\Leftrightarrow u\in\Sm(\Omega_1,\diprimo(\Omega_2)),\\
\WFbi{2}(u)&=\emptyset\Leftrightarrow u\in\Sm(\Omega_2,\diprimo(\Omega_1)).
\end{align*}
\end{pipi}
\begin{remark}
Note that $u\in\Sm(\Omega_1,\diprimo(\Omega_2))\bigcap\Sm(\Omega_2,\diprimo(\Omega_1))$ does not imply that $u\in\Sm(\Omega_1\times\Omega_2)$. Following \cite{Tre67}, a counterexample is, for instance, $\delta(x_1-x_2)$. The additional regularity needed such that $u\in\Sm(\Omega_1\times\Omega_2)$ is therefore, by Proposition \ref{pipi:regu}, encoded in $\WFbi{12}(u)$.
\end{remark}
\noindent The bisingular wave front set admits the following properties:
\begin{pipi}[\textbf{Properties of $\WFbi{}$}]
\label{pr:wfprop}
Let $u,\ v\in\dist(\Omega_1\times\Omega_2)$, $f\in\Sm(\Omega_1\times\Omega_2)$.
\begin{itemize}
	\item $\WFbi{}$ a closed set and is conic with respect to both covariables jointly.
	\item Let $\Omega_1=\Omega_2=\Omega$. Define for $f\in\Sm(\Omega \times \Omega)$ $Au(x_1,x_2)=u(x_2,x_1)$. Then we can define the pull-back $A^*$ as an endomorphism on $\dist(\Omega \times \Omega)$  by duality and we have $(x_1,x_2,\xi_1,\xi_2)\in\WFbi{}(Au)\Leftrightarrow (x_2,x_1,\xi_2,\xi_1)\in\WFbi{}(u)$.
	\item $\WFbi{}(u+v)\subset \WFbi{}(u)\cup\WFbi{}(v)$; $\WFbi{}(fu)\subset \WFbi{}(u)$.
\end{itemize}
\end{pipi}
\begin{remark}
These properties are quite similar to the ones the $SG$-wave front set of \cite{CM03} admits. This is not very surprising, as the bisingular calculus is formally very similar in its definition to the $SG$-calculus, through which the $SG$-wave front set is introduced. We will explore this connection in Section \ref{sec:sg}.
\end{remark}
\begin{lem}\label{lem:1microloc}
Let $C\in \mathrm{L}^{m_1,m_2}(\Omega_1\times\Omega_2)$, $u\in\dist(\Omega_1\times\Omega_2)$. Then we have $\WFbi{1}(Cu)\subset\WFbi{1}(u)$.
\end{lem}
\begin{proof}
Let $(x_1,x_2,\xi_1,0)\notin\WFbi{1}(u)$. By definition there exists a non-char (at $(x_1,\xi_1)$) $A\in\mathrm{L}_\mathrm{cl}^0(\Omega_1)$ such that $(A\otimes I)u\in\Sm(\Omega_1,\diprimo(\Omega_2)$. In particular, by Lemma \ref{lem:sminonevar} we have $\forall B\in  \mathrm{L}^{m,-\infty}(\Omega_1\times\Omega_2)$ that $B(A\otimes I)u$ is smooth. By the standard pseudodifferential calculus we can thus find an $A^\prime\in\mathrm{L}^0(\Omega_1)$ such that $A+A^\prime$ is elliptic in the sense of pseudodifferential operators and such that the symbol of $A^\prime$ vanishes on a conic neighbourhood $\Gamma$ of $(x_1,\xi_1)$. We thus have a parametrix $P\in\mathrm{L}^0(\Omega_1)$ such that $P(A+A^\prime)=I-R$ with $R\in\mathrm{L}^{-\infty}(\Omega_1)$. \\
Take $H\in\mathrm{L}^0(\Omega_1)$ such that $H$ is non-characteristic at $(x_1,\xi_1)$ and such that the symbol of $H$ vanishes outside a proper subcone of $\Gamma$. Then we have:
\begin{align*}
  (H\otimes I)Cu &= (H\otimes I)C\big((P(A+A^\prime)+R)\otimes I\big)u\\
	&= (H\otimes I)C(P\otimes I)(A\otimes I)u+(H\otimes I)C(PA^\prime\otimes I)u+\\
	&+(H\otimes I)C(R\otimes I)u\in\Sm(\Omega_1,\diprimo(\Omega_2)).
\end{align*}
\noindent The first summand is in $\Sm(\Omega_1,\diprimo(\Omega_2))$ due to the definition of $A$, the second as the symbol expansion given in Theorem \ref{thm:compos}, using the support properties of the symbols of $H$ and $A^\prime$, gives an operator in $L^{-\infty,0}$. The third one is in $\Sm(\Omega_1,\diprimo(\Omega_2))$ as $R\in\mathrm{L}^{-\infty,0}$ is already a smoothing operator in the first variable. This proves the claim.
\end{proof}
\begin{lem}\label{lem:12microloc}
Let $C\in \mathrm{L}^{m_1,m_2}(\Omega_1\times\Omega_2)$, $u\in\dist(\Omega_1\times\Omega_2)$. Then we have $\WFbi{12}(Cu)\subset\WFbi{12}(u)$.
\begin{proof}
Let $(x_1,x_2,\xi_1,\xi_2)\notin\WFbi{12}(u)$. Then, by definition, we know there exist $A_i \in \mathrm{L}^0_\mathrm{cl}(\Omega_i)$, non-char at $(x_i,\xi_i)$, $i=1,2$, such that 
\begin{align*} 
(A_1 \otimes A_2) u&\in\Sm(\Omega_1\times\Omega_2),\\ 
(A_1 \otimes I) u&\in\Sm(\Omega_1,\diprimo(\Omega_2)), \\
(I \otimes A_2) u&\in\Sm(\Omega_2,\diprimo(\Omega_1)).
\end{align*}
By the standard pseudodifferential calculus we can thus find $A_i^\prime\in\mathrm{L}^0(\Omega_i)$ such that $A_i+A_i^\prime$ is elliptic in the sense of pseudodifferential operators and such that the symbol of $A_i^\prime$ vanishes on a conic neighborhood $\Gamma_i$ of $(x_i,\xi_i)$. We then have two parametrices $P_i \in\mathrm{L}^0(\Omega_i)$ such that
\begin{equation*}
(P_1 \otimes P_2) \big((A_1 + A_1') \otimes (A_2 + A_2')\big) = I\otimes I - R_1 \otimes I - I \otimes R_2 - \underbrace{R_1 \otimes R_2}_{:=R},
\end{equation*}
with $R_i\in\mathrm{L}^{-\infty}(\Omega_i)$. Now pick $H_i\in\mathrm{L}^0(\Omega_i)$ such that $H_i$ is non-characteristic at $(x_i,\xi_i)$ and such that the symbol of $H_i$ vanishes outside a proper subcone of $\Gamma_i$. Then, recalling that by the standard pseudodifferential calculus if two operators have disjoint support their product is a smoothing operator, and using Lemma \ref{lem:sminonevar},
\begin{align*}
&(H_1 \otimes H_2)C u =  \\
&(H_1 \otimes H_2)C \bigg((P_1 \otimes P_2) \big((A_1 + A_1') \otimes (A_2 + A_2')\big) + R_1 \otimes I + I \otimes R_2 + R\bigg) u\\
&= \underbrace{(H_1 \otimes H_2)C (P_1 \otimes P_2) (A_1 \otimes A_2)u}_{\in \Sm \text{by eq \eqref{eq:wft31}}} + \underbrace{(H_1 \otimes H_2)C (P_1 \otimes P_2) (A_1' \otimes A_2)u}_{\in \Sm \text{by \eqref{eq:wft33} and by the support of}\, H_1, A'_1} \\
&+ \underbrace{(H_1 \otimes H_2)C (P_1 \otimes P_2) (A_1 \otimes A'_2)u}_{\in \Sm \text{by \eqref{eq:wft32} and by the support of}\, H_2, A'_2} + \underbrace{(H_1 \otimes H_2)C (P_1 \otimes P_2) (A_1' \otimes A'_2)u}_{\in \Sm \text{by the support of}\, H_1,H_2, A'_1,A_2'} \\
&+ (H_1 \otimes H_2)C (R_1 \otimes I)u + (H_1 \otimes H_2)C (I \otimes R_2) u + \underbrace{(H_1 \otimes H_2)C R u}_{\in \Sm \text{because}\, R\in\mathrm{L}^{-\infty,-\infty}}\\
&= (H_1 \otimes H_2)C (R_1 \otimes I)u + (H_1 \otimes H_2)C (I \otimes R_2) u \quad {\, \text{mod}\, \Sm.}
\end{align*}
Now, without loss of generality, we proceed with the calculations only for the term $(H_1 \otimes H_2)C (R_1 \otimes I)u$. We have
\begin{align*}
(H_1 \otimes H_2)C (R_1 \otimes I)u&= (H_1 \otimes H_2)C (R_1 \otimes (P_2(A_2 + A_2') + R_2))u \\
&= \underbrace{(H_1 \otimes H_2)C (R_1 \otimes P_2)(I \otimes A_2) u}_{\in \Sm \text{by \eqref{eq:wft33}}} \\
&+ \underbrace{(H_1 \otimes H_2)C (R_1 \otimes P_2)(I \otimes A_2')u}_{\in \Sm \text{by the support of}\, H_2, A'_2} \\
&+ \underbrace{(H_1 \otimes H_2)C (R_1 \otimes I)(I \otimes R_2) u}_{\in \Sm \text{because}\, R\in\mathrm{L}^{-\infty,-\infty}} \in \Sm, 
\end{align*}
therefore $(H_1 \otimes H_2)C u \in \Sm$. With similar computations, one can check that \
	\begin{align*} 
	(H_1 \otimes I) Cu&\in\Sm(\Omega_1,\diprimo(\Omega_2))\\ 
	(I \otimes H_2) Cu&\in\Sm(\Omega_2,\diprimo(\Omega_1))
  \end{align*} 
and this proves the claim.
\end{proof}
\end{lem}
\noindent The preceding Lemmas lead to the following proposition, which asserts that this definition of wave front set is indeed suitable for the calculus of bisingular operators:
\begin{pipi}[\textbf{Microlocality of bisingular operators}]\label{pr:microloc}
Let $C\in \mathrm{L}^{m_1,m_2}(\Omega_1\times\Omega_2)$, $u\in\dist(\Omega_1\times\Omega_2)$. Then we have $\WFbi{}(Cu)\subset\WFbi{}(u)$.
\end{pipi}
%

\subsection{Microelliptic properties of bisingular operators}


From the previous proposition, we can conclude that bielliptic operators preserve the bi-wave front set:
\begin{coro}
\label{coro:ell}
Let $A 	\in \mathrm{L}^{m_1,m_2}(\Omega_1,\Omega_2)$ be bi-elliptic. Then $\WFbi{}(Au) = \WFbi{}(u)$.
\begin{proof}
One inclusion follows directly form Proposition \ref{pr:microloc}. The other follows proceeding like in Proposition \ref{pipi:equal}.
\end{proof}
\end{coro}
\noindent Next we study the microellipticity properties of bisingular operators. For that we need a suitable definition of a characteristic set. As in Definition \ref{def:biell}, $\Omega_1$ and $\Omega_2$ are considered as compact manifolds.
\begin{deff}
Let $B \in \elbisi{m_1}{m_2}(\Omega_1,\Omega_2)$ and $v_0 = (x_1,\xi_1) \in \Omega_1 \times (\RR^{n_1} \setminus 0)$. We say that $B$ is not $1$-characteristic at $V := \pi_2^{-1} v_0 := \{(x_1,y,\xi_1,0) : y \in \Omega_2\}$ if 
\begin{enumerate}
\item for all $v \in V $ there exists an open conic neighbourhood $\Theta$ of $v$ such that $\sigma^{12} \neq 0$ on $\Theta \setminus (\RR^+ v)$,
\item $\sigma^1(B) \in \mathrm{L}^{m_2}_{\mathrm{cl}}(\Omega_2)$ is invertible with inverse in $\mathrm{L}^{-m_2}_{\mathrm{cl}}(\Omega_2)$ in an open conic neighbourhood $\Gamma$ of $v_0$.
\end{enumerate}
Let $\Charbi{1}(B)$ be the set of all $V$ such that $B$ is $1$-characteristic at $V$. \newline
\noindent We define $\Charbi{2}(B)$ accordingly for $W := \pi_2^{-1} w_0$, $w_0 \in \Omega_2 \times (\RR^{n_2} \setminus 0)$, by exchanging the roles of $\sigma^1(B)$ and $\sigma^2(B)$. \newline
\noindent Finally, we define $\Charbi{12}(B)$ as the set of points $z=(x_1,x_2,\xi_1,\xi_2)$, $|\xi_1||\xi_2| \neq 0$, where $\sigma^{12}(z) = 0$. \newline
\noindent Set
\begin{equation*}
\Charbi{}(B) := \Charbi{1}(B) \cup \Charbi{2}(B) \cup \Charbi{12}(B).
\end{equation*}  
\end{deff}
\begin{remark}
$B$ is bi-elliptic iff $\Charbi{}(B) = \emptyset$.
\end{remark}
\begin{remark}
\label{rem:wfchar}
With this notion Definition \ref{def:wft} can also be expressed in the form 
\begin{align*}
	\WFbi{1}(u)&=\bigcap_{\substack {A\in \mathrm{L}^0_\mathrm{cl}(\Omega_1) \\ (A\otimes I) u\in\Sm(\Omega_1,\diprimo(\Omega_2))}} \underbrace{\Charbi{1}(A\otimes I),}_{\Charcl(A)\times (\Omega_2\times(\RR^{n_2}\setminus 0))}
	\\
		\WFbi{2}(u)&=\bigcap_{\substack {A\in \mathrm{L}^0_\mathrm{cl}(\Omega_2) \\ (I\otimes A) u\in\Sm(\Omega_2,\diprimo(\Omega_1))}} \underbrace{\Charbi{2}(I\otimes A),}_{(\Omega_1\times(\RR^{n_1}\setminus 0)) \times \Charcl(A)}
	\\
	\WFbi{12}(u)&=\bigcap_{\substack {A_i\in \mathrm{L}^0_\mathrm{cl}(\Omega_1) \\ (A_1 \otimes A_2) u\in\Sm \\ (A_1\otimes I) u\in\Sm(\Omega_1,\diprimo(\Omega_2)) \\(I\otimes A_2) u\in\Sm(\Omega_2,\diprimo(\Omega_1))}} \Charbi{12}(A_1 \otimes A_2).
\end{align*}
\end{remark}
\noindent With the definition of $\Charbi{}$ we can review in the following Table \ref{tab:bisichar} the model cases of Table \ref{tab:exbisi}, setting $\RR_0^{n}=\RR^{n}\setminus 0$, $\Omega=\Omega_1\times\Omega_2$ and $\RR_0^{n_{12}}=  \RR_0^{n_1}\times\RR_0^{n_2}$.
\begin{table}[h]
\renewcommand{\arraystretch}{1.2}
\begin{center}
\begin{tabular}{|c|c|c|c|}
\hline 
Operator & $\Charbi{1}$ & $\Charbi{2}$ & $\Charbi{12}$  \\
\hline
$I\otimes I$ & $\emptyset$ & $\emptyset$ & $\emptyset$ \\ 
\hline 
$-\Delta_1 \otimes I+I\otimes (-\Delta_2)$ & $\Omega\times\RR_0^{n_1}\times\{0\}$ & $\Omega\times\{0\}\times\RR_0^{n_2}$ & $\Omega\times \RR_0^{n_{12}}$ \\ 
\hline 
$-\Delta_1 \otimes (-\Delta_2) $ & $\Omega\times\RR_0^{n_1}\times\{0\}$ & $\Omega\times\{0\}\times\RR_0^{n_2}$ & $\emptyset$ \\ 
\hline 
$-\Delta_1 \otimes (-\Delta_2+I) $  & $\Omega\times\RR_0^{n_1}\times\{0\}$  & $\emptyset$ & $\emptyset$ \\  
\hline 
$(-\Delta_1+I) \otimes (-\Delta_2+I) $  & $\emptyset$ & $\emptyset$& $\emptyset$ \\  
\hline 
$(-\Delta_1+I)^{-1} \otimes (-\Delta_2+I)^{-1} $ & $\emptyset$ & $\emptyset$ & $\emptyset$ \\ 
\hline 
\end{tabular} 
\end{center}
\label{tab:bisichar}
\caption{$\Charbi{}$ for model cases of bisingular operators}
\end{table}
\begin{lem}\label{lem:locpar}
Let $C 	\in \mathrm{L}^{m_1,m_2}(\Omega_1,\Omega_2)$ be such that
\begin{equation*}
\Charbi{1}(C) \cap (\Gamma \times \Omega_2 \times \{0\}) = \emptyset.
\end{equation*}
Let $a \in \mathrm{S}^{0}(\Omega_1)$ have support in a closed cone $\Gamma$ and be non-char (in the sense of $\Psi$DO) in $\Gamma^0$. Then there exists a $H \in \mathrm{L}^{-m_1,-m_2}(\Omega_1,\Omega_2)$ such that
\begin{equation*}
HC = A \otimes I - R,
\end{equation*}
where $R \in \elbisi{-\infty}{0}$.
\begin{proof}
The requirements on the support of $a$ mean precisely that $C$ is elliptic with respect to $a$ in the sense of \cite{Cor95}, Theorem 2.3.3. Therefore, by the classical calculus of pseudodifferential operators, we can find a symbol $e\in \bisi{-m_1}{-m_2}$ such that for all fixed $(x_2,\xi_2)$ the operator $E(x_2,\xi_2)=e(x_1,x_2,D_1,\xi_2)$ is a local parametrix with respect to $a$ namely,
\begin{equation*}
E(x_2,\xi_2)\, \sigma^2(C)(x_2,\xi_2) = R(x_2,\xi_2) + (A \otimes 1) (x_2,\xi_2),
\end{equation*}
where $R(x_2,D_2) \in \mathrm{L}^{-\infty,0}(\Omega_1,\Omega_2)$ and $E(x_2,D_2) \in \mathrm{L}^{-m_1,-m_2}(\Omega_1,\Omega_2)$. \newline
\noindent Now define $H$ as the operator with principal symbol
\begin{equation*}
\begin{split}
h = \psi_1(x_1,\xi_1) a_{m_1; \cdot} e_{m_1 ; \cdot} &+ \psi_2(x_2,\xi_2) a_{\cdot ; m_2} c^{-1}_{\cdot ; m_2} \\
& - \psi_1(x_1,\xi_1) \psi_2(x_2,\xi_2) a_{m_1; m_2} c^{-1}_{m_1 ; m_2}.
\end{split}
\end{equation*}
Using the calculus and Theorem \ref{thm:compos}, it is straightforward that $H$ matches the requirements.
\end{proof}
\end{lem}
\begin{lem}\label{lem:1microell}
Let $C\in \mathrm{L}^{m_1,m_2}(\Omega_1\times\Omega_2)$, $u\in\dist(\Omega_1\times\Omega_2)$. Then we have
\begin{equation*}
\WFbi{1}(u) \subseteq \Charbi{1}(C) \cup \WFbi{1}( C u).
\end{equation*}
\begin{proof}
Let $(x_1,x_2,\xi_1,0) \notin \Charbi{1}(C) \cup \WFbi{1}(C u)$. Then there exists $A \in \mathrm{L}^{0}(\Omega_1)$ such that we have $(A \otimes I) H C u \in \Sm(\Omega_1,\diprimo(\Omega_2))$, with $H$ as in Lemma \ref{lem:locpar}, due to microlocality (Proposition \ref{pr:microloc}) of $H$. Then we have
\begin{equation*}
\begin{split}
(A^2 \otimes I) u & = (A \otimes I) (R - HC) u \\
&= \underbrace{(A \otimes I) R}_{\in\, \mathrm{L}^{-\infty,0}} u - \underbrace{ (A \otimes I) HC u}_{\in\,\Sm(\Omega_1,\diprimo(\Omega_2))\, \text{by assumption on $A$}} \in \Sm.
\end{split}
\end{equation*}
Using Lemma \ref{lem:sminonevar} on the first summand, this proves the claim.
\end{proof}
\end{lem}
\noindent Using the previous Lemma, we conclude that
\begin{pipi}[\textbf{Microellipticity of bisingular operators with respect to the $1$ and $2$ components of $\WFbi{}$}]
Let $C \in \mathrm{L}^{m_1,m_2}(\Omega_1,\Omega_2)$, $u\in\dist(\Omega_1\times\Omega_2)$. Then
\begin{equation*}
\WFbi{i}(u) \subseteq \Charbi{i}(C) \cup \WFbi{i}( C u),
\end{equation*}
$i=1,2$.
\end{pipi}
\noindent We do, however, not obtain full microellipticity, i.e. with respect to the $\WFbi{12}$-component. This can be seen by the following example:
\begin{ex}
Consider $C=-\Delta\otimes(-\Delta)$, $u=\delta\otimes 1+1\otimes\delta$. Then $\Charbi{12}(C)=\emptyset$ and $Cu=0\in\Sm(\Omega_1\times\Omega_2)$, i.e. $\WFbi{12}(C u) = \emptyset$. But take any $A_1$ non-characteristic at $(0,\xi_1)$, $\xi_1\neq 0$. Then $(A_1\otimes I) u=(A_1\delta)\otimes 1$ which is never in $\Sm(\Omega_1,\diprimo(\Omega_2))$, and this means that $\WFbi{12}(u)$ is non-empty, because the \eqref{eq:wft32}- and, by a similar argument, \eqref{eq:wft33}-requirements fail to hold.
\end{ex}
\begin{remark}
The counterexample to full microellipticity could be circumvented by imposing stronger invertibility conditions in the definition of $\Charbi{12}$. This would, however, break the characterization of the wave front set of Remark \ref{rem:wfchar}, and lead to a loss of local information, while yielding no interesting cases not already covered by Corollary \ref{coro:ell}.
\end{remark}
\noindent With our definition we obtain, however, the following Lemma, which can be regarded as a microellipticity result for the \eqref{eq:wft31}-part of Definition \ref{def:wft}, for operators given by a tensor product.
\begin{lem}\label{lem:tenslem}
Let $C_i\in \mathrm{L}^{m_i}(\Omega_i)$, $u\in\dist(\Omega_1\times\Omega_2)$. If $$(x_1,x_2,\xi_1,\xi_2)\notin \Charbi{12}(C_1\otimes C_2) \cup \WFbi{12}( (C_1\otimes C_2) u)$$
there exist operators  $H_i\in \mathrm{L}^{0}(\Omega_i)$, non-characteristic at $(x_i,\xi_i)$, such that $(H_1\otimes H_2) u\in\Sm(\Omega_1\times\Omega_2)$.
\begin{proof}
By the standard pseudodifferential calculus we can pick $B_i\in \mathrm{L}^{-m_i}(\Omega_i)$ non-characteristic at $(x_i,\xi_i)$. Then the product $B_iC_i\in\mathrm{L}^{0}$ is non-characteristic at $(x_i,\xi_i)$. Proposition \ref{pr:microloc} and the definition of the bi-wave front set guarantee us the existence of $A_i\in\mathrm{L}^{0}$, non-characteristic at $(x_i,\xi_i)$, such that $(A_1\otimes A_2) (B_1C_1\otimes B_2C_2)u\in\Sm(\Omega_1\times\Omega_2)$. Therefore $H_i:=A_iB_iC_i$ fulfils the claim. 
\end{proof}
\end{lem}
\section{Comparison with $SG$ calculus} \label{sec:4}
\label{sec:sg}
\noindent In this section we will compare bisingular calculus with $SG$ calculus. $SG$ calculus is a global calculus obtained from the classical calculus by treating the variables and covariables equivalently by imposing on the symbols similar estimates as in bisingular calculus. These a priori formal similarities lead to interesting similarities in the calculus and in the analysis of singularities. However, the two calculi also differ in important aspects, as we will point out throughout the section.
\begin{deff}
A function $a(x,\xi) \in \Sm(\erre^{2n})$ is called a $SG$ symbol belonging to $\SG^{\mu, m}(\erre^n) := \SG^{\mu, m}$ iff for every $\alpha, \beta \in \zetapiu$ there exists a constant $C_{\alpha,\beta} > 0$ such that
\begin{equation*}
|D^\alpha_x D^\beta_\xi a(x,\xi)| \leq C_{\alpha, \beta}\normxi^{\mu - |\beta|}\normics^{m - |\alpha|}
\end{equation*}
for every $x,\xi \in \erre^n$.
A $SG$ pseudodifferential operator is an operator of the form
\begin{equation*}
\begin{split}
Au(x) := \int e^{ix\cdot\xi}\,a(x,\xi)\hat{u}(\xi)\,\dbar\xi,\,\,u \in \esse, 
\end{split}
\end{equation*} 
and the class of operators with symbols in $\SG^{\mu, m}$  is denoted by $\SGL^{\mu,m}$.
\end{deff}
\noindent A symbol $a\in \SG^{m_1,m_2}$ is called $\SG$ classical if it admits a homogeneous expansion with respect to $\xi$, for $|\xi|>>1$, a homogeneous expansion in $x$, for $|x|>>1$, and the two expansions satisfy certain compatibility conditions, we refer here to \cite{Cor95, ES97} for a precise definition of classical symbols. We limit our attention (in this context) to classical operators, i.e. such that their symbols are $SG$ classical.
\noindent As usual one proceeds to develop a calculus for these operators. As every classical $SG$ operator $A$ is also classical pseudodifferential operator, it admits a principal symbol $\sigma^\psi(A)$, homogeneous in the first variable. In addition, by exchanging the roles of the variables and covariables one obtains a symbol $\sigma^e(A)$, homogeneous in the second variable. The two satisfy a compatibility condition, i.e. that there is a third, bihomogeneous symbol $\sigma^{\psi e}(A)$, the leading term of the corresponding expansions of the $\psi$ and $e$-symbols. The \textit{principal homogeneous symbol} of the operator is then the couple $(\sigma^\psi(A),\sigma^e(A))$ which gives rise to the principal symbol 
$$\mathrm{Sym}_p(x,\xi)=\phi_{\psi}(\xi)\sigma^{\psi}(A)+\phi_e(x)\sigma^{e}(A)-\phi_\psi(x)\phi_e(\xi)\sigma^{\psi e}(A),$$
where $\phi_*$ are 0-excision functions. \\
\indent So far, this is very similar to the bisingular calculus, but the expansion formula for the symbol of a product is in fact a lot simpler. The operator compositions arising there are not present, and the composition formula is just the one corresponding to the $c^{12}$-term in Theorem \ref{thm:compos}.\\
This leads to a definition of ellipticity close to our notion of $12$-ellipticity, as no such thing as full invertibility of the symbols as operators is needed in the parametrix construction:
\begin{deff} 
A symbol $a \in \SG^{\mu, m}$ is $SG$-elliptic iff there exist constants $R, C_1,C_2 > 0$ such that
\begin{equation*}
 C_1 \normxi^\mu \normics^m \leq |a(x,\xi)| \leq C_2 \normxi^\mu \normics^m
\end{equation*}
when $|x|+|\xi| \geq R$.
\end{deff}
\noindent With this notion of ellipticity, we have the Fredholm property, i.e. an $SG$-elliptic operator admits a parametrix in the calculus.\\
Another important aspect to note is that in $SG$ calculus we have
\begin{pipi} 
Let $p \in \SG^{m, \mu}(\RR^n), q \in \SG^{r, \nu}(\RR^n)$. Then the commutator $[P,Q] := PQ - QP$ belongs to $\SGL^{m+r-1, \mu+\nu-1}(\RR^n)$.
\end{pipi}
\noindent while in the bisingular setting the analogon in fact does not hold true, see Corollary \ref{coro:comm} and Example \ref{ex:comm} which illustrates that the difference stems from an interaction of the $\Omega_1$ and $\Omega_2$-parts of the operators which is not present in the $SG$ case, where variables and covariables are independent.\newline
The notion of $SG$ calculus can be used to introduce a global analysis of singularities. It turns out that this exhibits interesting similarities with the above analysis of bisingular operators. In the following, we refer to \cite{Cor95}, \cite{CM03}; see also \cite{CJT13a}, \cite{CJT13b}. First, we introduce $SG$ characteristic sets and the $SG$ wave front set.

\begin{deff}[\textbf{$SG$ characteristic sets}]
Let $A \in \SGL^{m_1,m_2}$. Define the $SG$ characteristic set of $A$ as
\begin{equation*}
\Char_{\mathrm{SG}} (A) = \Char_{\mathrm{SG}}^{\psi}(A) \cup \Char_{\mathrm{SG}}^{e}(A) \cup \Char_{\mathrm{SG}}^{\psi e}(A),
\end{equation*}
where  
\begin{align*}
\Char_{\mathrm{SG}}^\psi(A) &= \{(x,\xi) \in \RR^n \times (\RR^n \setminus 0) : \sigma^\psi(A) (x, \xi) = 0\}, \\
\Char_{\mathrm{SG}}^e(A) &= \{(x,\xi) \in (\RR^n \setminus 0) \times \RR^n : \sigma^e(A) (x, \xi) = 0\}, \\
\Char_{\mathrm{SG}}^{\psi e}(A) &= \{(x,\xi) \in (\RR^n \setminus 0) \times (\RR^n \setminus 0) : \sigma^{\psi e}(A) (x, \xi) = 0\}.
\end{align*}
\end{deff}
\begin{deff}[\textbf{$SG$ wave front set}]
Let $u \in \esseprimo(\RR^n)$. Define the $SG$ wave front set of $u$ as
\begin{equation*}
\WFsg{} (u) = \WFsg{\psi}(u) \cup \WFsg{e}(u) \cup \WFsg{\psi e}(u),
\end{equation*}
where  
\begin{align*}
\WFsg{\psi}(u)&=\bigcap_{\substack {A\in \SGL^{0,0}(\RR^n) \\  A u\in\esse(\RR^n)}} \Char_{\mathrm{SG}}^\psi(A), \\
\WFsg{e}(u)&=\bigcap_{\substack {A\in \SGL^{0,0}(\RR^n) \\  A u\in\esse(\RR^n)}} \Char_{\mathrm{SG}}^e(A), \\
\WFsg{\psi e}(u)&=\bigcap_{\substack {A\in \SGL^{0,0}(\RR^n) \\  A u\in\esse(\RR^n)}} \Char_{\mathrm{SG}}^{\psi e}(A). 
\end{align*}
\end{deff}
\noindent This notion exhibits the following properties:
\begin{pipi}[\textbf{Properties of the $SG$ wave front set}]
Let $u , v \in \esseprimo(\RR^n)$, $f \in \esse(\RR^n)$. Then:
\begin{enumerate}
\item $\WFsg{}(u)$ is a closed set and $\WFsg{\psi}(u)$ is conical with respect to the variable $x$, $\WFsg{e}(u)$ with respect to the covariable $\xi$ and $\WFsg{}(u)$ with respect to both of them independently,
\item $(x,\xi)\in\WFsg{}(u)\Leftrightarrow(\xi,-x)\in\WFsg{}(\mathcal{F}u)$ \textit{(Fourier Symmetry)},
\item $\WFsg{}(u+v)\subseteq\WFsg{}(u)\cup\WFsg{}(v)$; $\WFsg{}(fu)\subseteq\WFsg{}(u)$,
\item $\WFsg{}(u)=\emptyset\Leftrightarrow u\in\esse(\RR^n)$ \textit{(Global regularity)}.
\end{enumerate}
\end{pipi}
\noindent Already we see the similarities, but also apparent differences, with the bisingular notion: 
\begin{itemize}
\item Fourier transformation (i.e. exchange of variables and covariables) corresponds to the exchange of variables $x_1$ and $x_2$ in Proposition \ref{pr:wfprop}.
\item The conical properties of the individual components of the wave front sets correspond to the homogeneity properties of the corresponding principal symbol part.
\item The global (i.e.) $\esse$-regularity is of course due to the fact that $SG$-calculus imposes bounds on the variables also.
\end{itemize}
Moreover, $SG$ operators satisfy $SG$ microlocalty and microellipticity, i.e. (just as in the bisingular case):
\begin{pipi}[\textbf{$SG$ Microlocality and $SG$ Microellipticity}]
Let $u \in \esseprimo(\RR^n)$ and $C \in \SGL^{m,\mu}$. Then we have the double inclusion 
\begin{equation*}
\WFsg{}( C u)\subseteq\WFsg{}(u) \subseteq \Char_{\mathrm{SG}}(C) \cup \WFsg{}( C u).
\end{equation*}
\end{pipi}
\noindent This again stems from an individual double inclusion with respect to each wave front set component.
A capital difference lies, however, in the structure of the wave front set. For the components of the $SG$ wave front set have
\begin{equation*}
\begin{split}
\WFsg{}(u)\subset \big(\RR^{n}\times (\RR^{n}\setminus 0)\big)\cup\big((\RR^{n}\setminus 0)\times \RR^{n}\big)\cup\big((\RR^{n}\setminus 0)\times (\RR^{n}\setminus0)\big).
\end{split}
\end{equation*}
Here, the $\big(\RR^n\times(\RR^n\setminus 0)\big)\ni(x,\xi)$ component corresponds exactly to singularities at finite arguments $x$ with propagation direction $\xi$, the $\big(\RR^n\times(\RR^n\setminus 0)\big)$ component yields the same interpretation in the Fourier transformed space (growth singularities of $u$ become differential singularities of $\hat u$) and the $\big((\RR^n\setminus 0)\times(\RR^n\setminus 0)\big)$ component corresponds to high oscillations present at infinite arguments.\\
In the bisingular case, new phenomena are present. The $12$-component has the classical interpretation, in the sense that it includes all the `classical' singularities (see Lemma \ref{lem:wfcl}), but the other components lose some amount of localization. This is reflected in the structure of the ($1$- and $2$-components) of the bi-wave front set. In fact,
\begin{equation*}
\begin{split}
\WFbi{}(u)\subset (\Omega_1\times\Omega_2)\times(\RR^{n_1+n_2}\setminus 0)=\ &(\Omega_1\times\Omega_2)\times\big((\RR^{n_1}\setminus 0)\times \{0\}\big)\\
\sqcup\ &(\Omega_1\times\Omega_2)\times\big(\{0\}\times(\RR^{n_2}\setminus 0)\big)\\
\sqcup\ &(\Omega_1\times\Omega_2)\times\big((\RR^{n_1}\setminus 0)\times(\RR^{n_2}\setminus 0)\big),
\end{split}
\end{equation*}
where if e.g. for one $x_1$ we have $(x_1,x_2,0,\xi_2)$ present in the wave front set, all $(y,x_2,0,\xi_2)$ are present in the wave front set. This is due to the fact that bi-ellipticity involves true invertibility, i.e. a non-local requirement.\\
Another difference arises as follows: the $1$ and $2$-component can be understood as the boundary faces of the $12$-component, whereas in the $SG$ case the $\psi e$-component is interpreted as the corner of the wave front space where the $e$- and $\psi$-component meet, i.e. the roles as boundaries are interchanged, see Figure \ref{fig:sgbicomparison}.\\
\begin{figure}[!ht]
\begin{center}
\begin{tikzpicture}[
        scale=1.0,
        MyPoints/.style={draw=black,fill=white,thick},
        Segments/.style={draw=black!50!red!70,thick},
        ]
    \draw[->] (-0.5,0)--(4.5,0) node[right]{$x$};
    \draw[->] (0,-0.5)--(0,4.5) node[above]{$\xi$};

    \draw[->] (5.5,0)--(10.5,0) node[right]{$\xi_1$};
    \draw[->] (6,-0.5)--(6,4.5) node[above]{$\xi_2$};
    
	\node at (2,-0.8) {$\WFsg{}$};    
	\node at (8,-0.8) {$\WFbi{}$}; 
    
    \draw[-] (0,4)--(4,4) node[above,midway] {$\WFsg{\psi}$};
    \draw[-] (4,0)--(4,4) node[right,midway] {$\WFsg{e}$};
    
    \draw[-] (6,4)--(10,4) node[above right]{$\WFbi{12}$};
    \draw[-] (10,0)--(10,4);    

    \coordinate (A) at (4,4);
    \coordinate (B) at (6,4);
    \coordinate (C) at (10,0);

    \fill[MyPoints] (A) circle (0.8mm) node[above right]{$\WFsg{\psi e}$};
    \fill[MyPoints] (B) circle (0.8mm) node[above right]{$\WFbi{2}$};
    \fill[MyPoints] (C) circle (0.8mm) node[above right]{$\WFbi{1}$};
\end{tikzpicture}
\caption{A schematic comparison of the components of $\WFsg{}$ and $\WFbi{}$}
\label{fig:sgbicomparison}
\end{center}
\end{figure}
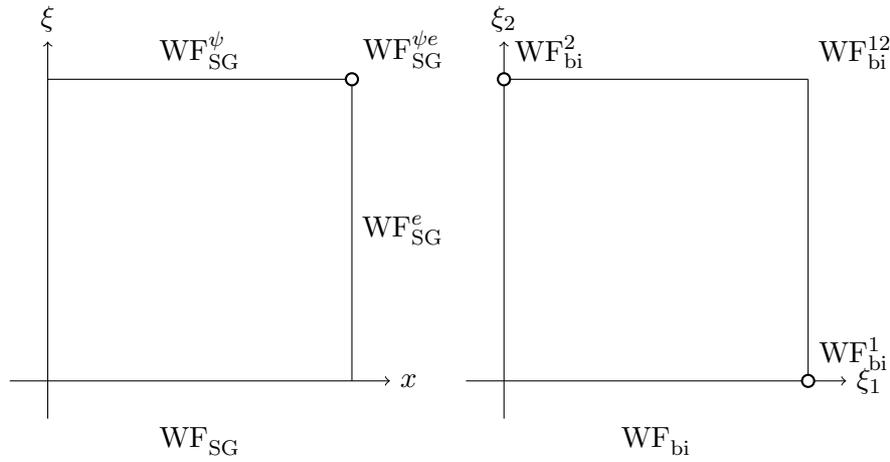
\noindent This has the following consequence:
\begin{ex}
Consider the one dimensional case. Following here Example 2.7. in \cite{CM03}, there exists a distribution $u(x) = e^{i x^2/2}$, $x \in \RR$, such that $\WFsg{\psi} (u) = \emptyset = \WFsg{e}(u)$ and $\WFsg{\psi e} (u) = (\RR \setminus 0) \times (\RR \setminus 0)$. \newline
\noindent However, there cannot exist a distribution $v \in \eeprimo(\Omega_1 \times \Omega_2)$, $\Omega_1, \Omega_2 \subset \RR$, such that $\WFbi{1} (v) = \emptyset = \WFbi{2} (v)$ and $\WFbi{12} (v) = \Omega_1 \times \Omega_2 \times (\RR\setminus 0) \times (\RR \setminus 0)$. This is because $\WFbi{12} (v) = \Omega_1 \times \Omega_2 \times (\RR\setminus 0) \times (\RR \setminus 0)$ is an open set, and $\WFbi{}(v)$ has to be closed. \newline
Similarly $\tilde v= \delta \otimes 1$ has $\WFbi{1}(\tilde v) = (\{0\} \times \Omega_2)\times ((\RR\setminus 0)\times \{0\})$, $\WFbi{12}(\tilde v)= \emptyset =\WFbi{2}(\tilde v)$, but there cannot exist a distribution $\tilde u$ such that $\WFsg{e} (\tilde u)=\RR\times(\RR\setminus 0)$ but $\WFsg{\psi e} (\tilde u)=\emptyset$, again due to closedness.
\end{ex}
\bibliography{Bibliografia}
\bibliographystyle{amsalpha}
\end{document}